\documentclass[11pt]{amsart}
\usepackage[left=2cm,top=2cm,right=3cm,nofoot]{geometry}
\usepackage{amsmath,mathtools}%
\usepackage{amsfonts}%
\usepackage{amssymb}%
\usepackage{graphicx}
\usepackage{esint}
\usepackage{hyperref}
\usepackage{enumitem}
\usepackage{accents}
\usepackage{etoolbox}
\patchcmd{\subsection}{-.5em}{.5em}{}{}
\newtheorem{theorem}{Theorem}[section]
\theoremstyle{plain}

\newtheorem{definition}[theorem]{Definition}

\newtheorem{lemma}[theorem]{Lemma}

\newtheorem{remark}[theorem]{Remark}

\numberwithin{equation}{section}
\theoremstyle{plain}
\newtheorem{claim}{Claim}
\usepackage{etoolbox}
\AtEndEnvironment{proof}{\setcounter{claim}{0}}

\begin{document}

\title[]{Zero modes on product Riemannian manifolds}

\author{Jurgen Julio-Batalla}
\address{ Universidad Industrial de Santander, Carrera 27 calle 9, 680002, Bucaramanga, Santander, Colombia}
\email{ jajuliob@uis.edu.co}
\thanks{ }

\begin{abstract} 
This paper is concerned with the zero mode equation $D_g\varphi=iA\cdot\varphi$ on product of closed spin manifolds $(M_1^{n_1}\times M_2^{n_2},g_1+g_2,\sigma)$ of dimensions $n_1\leq n_2$ respectively. Here $A$ is a real vector field on $M^n=M_1^{n_1}\times M_2^{n_2}$. Under non-increasing condition on $|\varphi|$ we prove that $$\parallel A\parallel_n^2\geq\frac{n_2}{4(n_2-1)}Y(M^n,[g]),$$ where $Y(M^n,[g])$ is the Yamabe constant of $(M^n,g)$. This estimate is sharp in even dimensions. We also obtain a similar estimate for non trivial solutions of the zero mode type equation $D_g\varphi=f\varphi$, where $f$ is a scalar function.
 
\end{abstract}

\maketitle

\section{Introduction}
On the product of closed spin manifolds $(M^n=M_1^{n_1}\times M_2^{n_2},g=g_1+g_2,\sigma)$ we are interested in the zero mode equation
\begin{equation}\label{Zero}
D_g\varphi=iA\cdot\varphi,
\end{equation}

where $A$ is a real vector field, $\varphi$ is a spinor field on $(M^n,g)$ and $D_g$ is the Dirac operator respect to the metric $g$. The action ``$\cdot$" of $A$ is via the Clifford multiplication.

This equation appears in various contexts in mathematical physics. For instance in the study of the stability of an atom interacting with a spin-magnetic field (see \cite{Frohlich}). In such settings, the existence of a  non-trivial pair $(\varphi,A)$ satisfying \eqref{Zero} leads to obstructions issues.

The first example of a zero mode pair $(\varphi,A)$ was constructed by Loss-Yau in \cite{Loss-Yau} in $\mathbb{R}^3$. Since then, considerable effort has been devoted to understanding the structure of the equation \eqref{Zero}. For an overview of recent developments on zero mode equation, we refer the reader to   \cite{Frank-Loss1} and the references therein.

A significant advancement was made by Frank-Loss in \cite{Frank-Loss2}, who studied equation \eqref{Zero}  in $\mathbb{R}^n$ with respect to the Euclidean metric. They found a necessary condition for a non-trivial solution $(\varphi, A)$. More precisely, if $(\varphi,A)$ is a non-trivial solution then 
\begin{equation}\label{FrankInequality}
\parallel A\parallel_n^2\geq \frac{n}{n-2}S_n,
\end{equation}

where $S_n$ is the optimal Sobolev constant in $\mathbb{R}^n$. 

Examples achieving equality in \eqref{FrankInequality} include the explicit pairs from \cite{Loss-Yau}  in dimension $n=3$ and their generalizations to higher odd dimensions found by Dunne-Mi in \cite{Dunne}. Moreover, Frank-Loss  in \cite{Frank-Loss2} proved that these examples are  the only solutions for which the equality holds, showing that \eqref{FrankInequality} is not optimal in even dimensions.

Recently, this result has been extended to the setting of general spin manifolds. J. Reuss in \cite{Reus} adapted the method of Frank-Loss and obtained that 
\begin{equation}\label{ReusTheorem}
\parallel A\parallel_n^2\geq \frac{n}{4(n-1)}Y(M^n,[g]),
\end{equation}

where $Y(M^n,[g])$ is the Yamabe constant of $(M^n,g)$.

Very recently, Wang-Zhang in \cite{Wang} provided a complete classification of the equality case. They proved that equality holds if and only if $(M^n,g)$ is a Sasaki-Einstein manifold. 

From this characterization, a natural question arises: under what geometric or analytic conditions can the sharp inequality \eqref{ReusTheorem} be improved?

One remarkable feature of the zero mode equation \eqref{Zero}  that helps address this question is its conformal invariance. Specifically, if $(\varphi,A)$ solves \eqref{Zero} for a metric $g$, then a corresponding pair exists for any conformal metric in $[g]$, with the same $\parallel\cdot \parallel_n$ norm of the vector potential. Thus choosing a suitable metric in a given conformal class may lead to improve estimates for the $L^n$-norm of $A$.

To state precisely our main theorem, we introduce the following definition:

\begin{definition}
For any given couple of real valued functions $F,G$ we say that $F$ is non-increasing with respect to $G$ if $$\int_M\left\langle\nabla F,\nabla ln(u_1)\right\rangle dv_g\leq 0,$$
where $u_1$ is the first eigenfunction of the weighted linear problem $$-a_n\Delta u+s_gu=kG^2u$$
or equivalently, $u_1$ is the minimum of $$L(u)=\dfrac{\int_M\left(a_n|\nabla u|^2+s_gu^2\right)dv_g}{\int_MG^2u^2dv_g}$$
over the Sobolev space $H^1(M)$. Here $a_n=4(n-1)/(n-2)$.
\end{definition}

With this definition in hand, our main result improves the lower bound of $\parallel A\parallel_n^2$ when  the product metric $g=g_1+g_2$ admits a zero mode $(\varphi,A)$ such that the $g-$length of $\varphi$ is non-increasing with respect to $|A|_g$. 

\begin{theorem}
Let $(M^n,g,\sigma)$ be a closed Riemannian spin manifold of dimension $n=n_1+n_2$ which is the product of closed Riemannian manifolds $(M_1^{n_1},g_1)$ ,$(M_2^{n_2},g_2)$ with dimensions $n_1\leq n_2$ respectively. Assume that the product metric $g=g_1+g_2$ has positive scalar curvature. Let $(\varphi,A)$ be a non trivial solution of \eqref{Zero} for the metric $g$.

If $|\varphi|^2_g$ is non increasing with respect to $|A|_g$ then 
\begin{equation}\label{inequality1}
\parallel A\parallel_n^2\geq \frac{n_2}{4(n_2-1)}Y(M^n,[g]).
\end{equation}

Moreover, if the equality holds then: 
\begin{itemize}
\item [(a)] The vector potential $A$ has constant length,
\item [(b)] The metric $g$ is a Yamabe metric,
\item [(c)] $(M_2^{n_2},g_2)$ admits a non trivial real Killing spinor,
\item [(d)] $(M_1^{n_1},g_1)$ admits a parallel spinor when $n_1<n_2$
or either a parallel spinor or a real Killing spinor when $n_1=n_2$.
\end{itemize}

\end{theorem}

\begin{remark}
The non-increasing condition on $|\varphi|^2_g$ with respect to $|A|_g$ is a slight modification of the condition that $|\varphi|_g^2$ is non-increasing along the gradient flow generated by the first eigenfunction $u_1$ of the conformal Laplacian $L_g$ with respect to the weight $|A|_g^2$. In particular when $|\varphi|_g$ or $u_1$ is a constant function, this condition is clearly satisfied.
\end{remark}

\begin{remark}
There are explicit examples of zero modes that fall  in the equality of \eqref{inequality1}. For instance, in dimension $n=3$ the zero mode $(\varphi,A)$ found by Loss-Yau in \cite{Loss-Yau} satisfies  $$|\varphi(x)|=\frac{1}{1+|x|^2},\quad |A(x)|=\frac{3}{1+|x|^2}$$
with respect to the Euclidean metric $dx^2$.
By the conformally invariance of the zero mode equation \eqref{Zero} the pair  $(\varphi,A)$ induces a solution $(\varphi_0,A_0)$ of the equation on the sphere $(\mathbb{S}^3,g_0)$ with respect to the round metric $$g_0=\frac{4}{\left(1+|x|^2\right)^2}dx^2.$$
Consequently, $$|\varphi_0|_{g_0}=\frac{1}{2},\quad |A_0|_{g_0}=\frac{3}{2}.$$
Therefore the pair $(\varphi_0\otimes\varphi_0,A_0\oplus A_0)$ is a zero mode on the product $(\mathbb{S}^3\times\mathbb{S}^3,g_0+ g_0)$ such that $|\varphi_0\otimes\varphi_0|^2$ is constant (hence non increasing along $|A_0\oplus A_0|$). On the other hand,
$$\left(\int_{\mathbb{S}^3\times\mathbb{S}^3}|A_0\oplus A_0|^6dv_{g_0+g_0}\right)^{1/3}=2|A_0|^2_{g_0}vol(\mathbb{S}^3\times\mathbb{S}^3,g_0+g_0)^{1/3}=\frac{3}{8}Y(\mathbb{S}^3\times\mathbb{S}^3,[g_0+g_0]).$$
Hence the zero mode $(\varphi_0\otimes\varphi_0,A_0\oplus A_0)$ satisfies the equality in \eqref{inequality1}.

Similarly, using the zero modes obtained by Dunne-Mi in \cite{Dunne} we can find examples of zero modes on products of odd-dimensional spheres $(\mathbb{S}^{2n+1},g_0)$ that satisfy  the equality in the previous Theorem.

Another family of examples can be obtained by considering the product of an odd-dimensional sphere $(\mathbb{S}^{2n+1},g_0)$ with a circle $(\mathbb{S}^1,dt^2)$.

\end{remark}

\begin{remark}

The previous examples of zero mode $(\varphi,A)$ that achieve the equality in \eqref{inequality1} occur in even dimension. This contrasts with the sharp inequalities \eqref{FrankInequality} and \eqref{ReusTheorem} where we know that the equality case only occurs in odd dimensions (see \cite{Frank-Loss2,Wang}).
\end{remark}

In this paper we also consider a closely related equation (which we called the zero mode type equation)
\begin{equation}\label{zero-type}
D_g\varphi=f\varphi,
\end{equation}
where $f$ is a scalar function.

We obtain a very similar necessary condition for the existence of non-trivial solutions of \eqref{zero-type}:

\begin{theorem}
Let $(M^n,g,\sigma)$ be a closed Riemannian spin manifolds as the Theorem 1.2. Let $(\varphi,f)$ be a non trivial solution of \eqref{zero-type} for the product metric $g=g_1+g_2$.

If $|\varphi|_g^2$ is non-increasing with respect to $f$ then
\begin{equation}\label{inequality2}
\parallel f\parallel_n^2\geq \frac{n_2}{4(n_2-1)}Y(M^n,[g]).
\end{equation}
Moreover, the equality holds if and only if:
\begin{itemize}
\item [(a)] The scalar potential $f$ is constant,
\item [(b)] the metric $g$ is a Yamabe metric,
\item [(c)] $(M_2^{n_2},g_2)$ admits a non trivial real Killing spinor,
\item [(d)] $(M_1^{n_1},g_1)$ admits a parallel spinor when $n_1<n_2$
or either a parallel spinor or a real Killing spinor when $n_1=n_2$.
\end{itemize}
\end{theorem}

\begin{remark}
In contrast to the inequality \eqref{inequality1}, we completely characterize the equality case in the sharp inequality \eqref{inequality2}. 
\end{remark}

Finally, regarding regularity, due to regularity results in \cite{Frank-Loss1,Frank-Loss2,Wang} for zero modes $(\varphi,A)$, we assume throughout the paper that $|A|$ is in $L^n(M^n)$ and the spinor field $\varphi$ is in $L^p(M^n)$ for some $p>n/(n-1)$.

The plan of this paper is as follows. In section 2 we will recall some background material on spin manifolds and Penrose-type operators on product manifolds. In section 3 we will prove the inequality \eqref{inequality1}. The proof of the inequality \eqref{inequality2} is carried out in section 4.

\section{Preliminaries}
\subsection{Spin manifolds}
On a closed oriented Riemannian manifold $(M^n,g)$ we can define a $SO(n)$-principal bundle $P_{SO}M$ over $M$ of oriented $g-$orthonormal bases at $x\in M$. For $n\geq 3$, there exists the universal covering $\sigma:spin(n)\rightarrow SO(n)$ where $spin(n)$ is the group generated by even unit-length vector of $\mathbb{R}^n$ in the real Clifford algebra $Cl_n$ (the associative $\mathbb{R}-$algebra generated by relation $VW+WV=-2( V,W)$ for the Euclidean metric $(,)$). The manifold $M$ is called spin if there is a $spin(n)-$principal bundle $P_{spin}M$ over $M$ such that it is a double covering of $P_{SO}M$ whose restriction to each fiber is 
the double covering $\sigma:spin(n)\rightarrow SO(n)$. Such a double covering from $P_{spin}M$ to $P_{SO}M$, $\sigma$, is 
known as a spin structure. 

\medskip

There are  four special structures associated to a spin manifold $(M^n,g,\sigma)$:
\begin{enumerate}
\item A complex vector bundle $SM:=P_{spin}(M)\times_{\rho}\Sigma_n$ where $\rho:spin(n)\rightarrow Aut(\Sigma_n)$ is the restriction to $spin(n)$ of an irreducible representation $\rho:\mathbb{C}l_n\rightarrow End(S_n) $ of the complex Clifford algebra $\mathbb{C}l_n\simeq Cl_n\otimes_{\mathbb{R}} \mathbb{C}$,  $\Sigma_n\simeq\mathbb{C}^N$ and $N=2^{[n/2]}$.
\item The Clifford multiplication $m$ on $SM$ defined by
\begin{align*}
m:TM\times SM&\rightarrow SM\\
X\otimes\varphi&\mapsto X\cdot_g\varphi:=\rho(X)\varphi.
\end{align*} 
\item A Hermitian product $\langle\cdot,\cdot\rangle$ on sections of $SM$.

\item A Levi-Civita connection $\nabla$ on $SM$.
\end{enumerate}

All these structures are compatible in the following sense:
\begin{align*}
\langle X\cdot\varphi,\psi\rangle&=-\langle\varphi,X\cdot\psi\rangle, \\
X(\langle\varphi,\psi\rangle)&=\langle\nabla_X\varphi,\psi\rangle+\langle\varphi,\nabla_X\psi\rangle,\\
\nabla_X(Y\cdot\varphi)&=\nabla_XY\cdot\varphi+Y\cdot\nabla_X\varphi,
\end{align*}
for all $X,Y\in\Gamma(TM)$ and $\varphi,\psi\in\Gamma(SM)$.
Given the Levi-Civita  connection $\nabla:\Gamma(SM)\rightarrow\Gamma(Hom(TM,SM))$ and identifying $\Gamma(Hom(TM,SM))$ with $\Gamma(TM\otimes SM)$, we can define the Dirac operator $D_g$ as the composition of $\nabla$ with the Clifford multiplication $m$ i.e. $D_g:=m\circ \nabla$. For a local orthonormal frame $\{E_i\}$ we have $$D_g\varphi=\sum\limits_{i=1}^nE_i\cdot_g\nabla_{E_i}\varphi.$$

\subsection{Penrose-type operator on products}

We will recall some properties of a Penrose-type operator induced by the decomposition of $TM=TM_1\oplus TM_2$  and their relation with the Dirac operator. Details can be found in \cite{Alex}

Let $E_{1,1},E_{1,2},\cdots,E_{1,n_1},E_{2,1},E_{2,2},\cdots E_{2,n_2}$ be an adapted local orthonormal frame on $TM_1\oplus TM_2$ and let $\pi_i:TM\rightarrow TM_i$ be the orthogonal projection of $TM$ on $TM_i$ for $i=1,2$.

We can define the operator $$D_{(i)}\varphi:= \sum_{j=1}^{j=n_i}E_{i,j}\cdot\nabla_{\pi_iE_{i,j}} \varphi,$$
for $i=1,2$.

It is straightforward to prove the following properties of $D_{(i)}$

\begin{lemma}
\begin{itemize}
\item [(1)] $D=D_{(1)}+D_{(2)}$;
\item [(2)]$D_{(1)}D_{(2)}+D_{(2)}D_{(1)}=0$;
\item [(3)]$D^2=D^2_{(1)}+D^2_{(2)}$;
\item [(4)]$D_{(i)}$ is self adjoint (for $i=1,2$). 
\end{itemize}

\end{lemma}

In this setting, it is defined the Penrose-type operator $$T_X\varphi=\nabla_X\varphi +\frac{1}{n_1}\pi_1(X)\cdot D_{(1)}\varphi+\frac{1}{n_2}\pi_2(X)\cdot D_{(2)}\varphi,$$
for all spinor field $\varphi$ and all vector field $X$.

This operator satisfies that \begin{equation}\label{decompositionT}
|T\varphi|^2=|\nabla\varphi|^2-\frac{1}{n_1}|D_{(1)}\varphi|^2-\frac{1}{n_2}|D_{(2)}\varphi|^2
\end{equation}

\section{Proof of Theorem 1.2}
In this section we will prove the main Theorem 1.2. We divide the Theorem in two parts.
First, we deal with the inequality \eqref{inequality1}. After, we will discuss the necessary conditions for which the equality holds in \eqref{inequality1}. 

The first part of Theorem 1.2 is equivalent to

\begin{theorem}
Let $(M^n=M_1^{n_1}\times M_2^{n_2},g=g_1+g_2,\sigma)$ be a product of closed spin manifolds of dimensions $n_1\leq n_2$. Assume that the scalar curvature $s_g$ is positive. Let $(\varphi,A)$ be a non trivial solution of the zero mode equation $$D_g\varphi=i\lambda A\cdot_g\varphi,\quad\int_M |A|^ndv_g=1. $$
If $|\varphi|^2$ is non-increasing with respect to $|A|$ then $$\lambda^2\geq \frac{n_2}{4(n_2-1)}Y(M^n,[g]).$$
\end{theorem}

\begin{proof}
By the Schrödinger-Lichnerowitz formula $$0=\int_M\left(- \left\langle D^2_g\varphi,\varphi\right\rangle+\frac{s_g}{4}|\varphi|^2+|\nabla \varphi|^2\right)dv_g$$
Using Stoke Theorem and the zero mode equation we have $$0=\int_M\left(- \lambda^2|A|^2|\varphi|^2+\frac{s_g}{4}|\varphi|^2+|\nabla \varphi|^2\right)dv_g.$$
Taking the Penrose-type operator $T$ from the section 2, we decompose the term $|\nabla\varphi|^2$ as \eqref{decompositionT} in order to get
$$0=\int_M\left( - 4\lambda^2|A|^2+s_g\right)|\varphi|^2+\int_M 4|T\varphi|^2+\int_M\left(\frac{4}{n_1}|D_{(1)}\varphi|^2+\frac{4}{n_2}|D_{(2)}\varphi|^2\right).$$
Since each $D_{(i)}$ is self adjoint and $D^2=D^2_{(1)}+D^2_{(2)}$ we see $$\int_M|D_{(2)}\varphi|^2=\int_M|D\varphi|^2-\int_M|D_{(1)}\varphi|^2.$$
Therefore, from the two last equations we obtain
 
$$0=\int_M |\varphi|^2\left(4\left(\frac{1-n_2}{n_2}\right)\lambda^2|A|^2+s_g\right)+\int_M 4|T\varphi|^2+\int_M 4\left( \frac{1}{n_1}-\frac{1}{n_2}\right)|D_{(1)}\varphi|^2.$$
On the other hand, for any positive function $u$ we have
$$\int_Ma_n\left\langle\nabla ln(u),\nabla |\varphi|^2\right\rangle=\int_M-a_n|\varphi|^2\left(\frac{\Delta u}{u}-\frac{|\nabla u|^2}{u^2}\right).$$

Adding the last two equations we get the following integral identity (*)
\begin{align*}
0=\int_M\frac{|\varphi|^2}{u}\left(-a_n\Delta u+s_gu+4\left(\frac{1-n_2}{n_2}\right)\lambda^2|A|^2u\right)+\int_M 4|T\varphi|^2+\int_M 4\left( \frac{1}{n_1}-\frac{1}{n_2}\right)|D_{(1)}\varphi|^2\\
+\int_M a_n|\varphi|^2\frac{|\nabla u|^2}{u^2}-\int_Ma_n\left\langle\nabla ln(u),\nabla |\varphi|^2\right\rangle.
\end{align*}
Now consider the constant $I(M,g,|A|^2)$ defined by $$I(M,g,|A|^2)=\inf_{H^1(M^n)}\frac{\int_M\left( a_n|\nabla u|^2+s_gu^2\right)dv_g}{\int_M (|A|u)^2dv_g}.$$
i.e. the first eigenvalue of the linear weighted problem $-a_n\Delta u+s_gu=I|A|^2u.$

Let $u_1$ be the positive smooth function which realizes $I(M,g,|A|^2)$. By the non increasing assumption  on $|\varphi|^2$, we have $$-\int_Ma_n\left\langle\nabla ln(u_1),\nabla |\varphi|^2\right\rangle\geq 0.$$
In particular, inserting $u_1$ in the integral identity (*) we conclude that $$4\left(\frac{n_2-1}{n_2}\right)\lambda^2< I(M,g,|A|^2)$$
implies that $\varphi$ must be trivial.
Hence  $$4\left(\frac{n_2-1}{n_2}\right)\lambda^2\geq I(M,g,|A|^2)$$
To obtain a lower bound for $I$ note from normalization of $|A|$ and Hölder's inequality that $$\int_M (|A|u)^2dv_g\leq \left(\int_Mu^{2n/(n-2)}dv_g\right)^{(n-2)/n},$$
and hence

 $$I\geq \inf_{H^1(M)} \frac{\int_M\left( a_n|\nabla u|^2+s_gu^2\right)dv_g}{\left(\int_Mu^{2n/(n-2)}dv_g\right)^{(n-2)/n}}=Y(M^n,[g]). $$ 
Therefore $$\lambda^2\geq \frac{n_2}{4(n_2-1)}Y(M^n,[g]).$$
\end{proof}

To end this section we discuss a necessary condition to the equality in \eqref{inequality1}.

If $$\lambda^2=\frac{n_2}{4(n_2-1)}Y(M^n,[g])$$ then
$$\lambda^2=\frac{n_2}{4(n_2-1)}I(M,g,|A|^2)$$ and $I(M,g,|A|^2)=Y(M,[g])$.

The equation $\lambda^2=\frac{n_2}{4(n_2-1)}I(M,g,|A|^2)$ and the integral identity (*) imply that $T\varphi=0,\quad D_1\varphi=0$ (if $n_1<n_2$) and $u_1$ is a constant.

\begin{claim}\label{Claim1}
The equality $I(M,g,|A|^2)=Y(M,[g])$  and the fact that $u_1$ is constant occur if and only if  $|A|$ is a constant and $g$ is a Yamabe metric.
\end{claim}
\begin{proof}[Proof of the claim]
The equality $I(M,g,|A|^2)=Y(M,[g])$  and $u_1$ constant imply that $g$ must be a Yamabe metric. In particular $s_g$ is constant. Since $s_g=I(M,g,|A|^2)|A|^2$ we get that $|A|$ is constant.

Now if $g$ is a Yamabe metric and $|A|$ is constant then
$|A|=1/(vol(M,g))^{1/n}$ (from normalization  $\int_M |A|^ndv_g=1$).

Thus $$I(M,g,|A|^2)=\inf_{H^1}\frac{\int_M\left( a_n|\nabla u|^2+s_gu^2\right)dv_g}{\int_M (|A|u)^2dv_g}= \inf_{H^1}\frac{\int_M\left( a_n|\nabla u|^2+s_gu^2\right)dv_g}{\int_M u^2dv_g}vol(M,g)^{2/n}.$$
On the other hand $$\lambda_1(L_g)=\inf_{H^1}\frac{\int_M\left( a_n|\nabla u|^2+s_gu^2\right)dv_g}{\int_M u^2dv_g}$$
is the first eigenvalue of the conformal Laplacian $L_g$. 

Since that $g$ is a Yamabe metric we have $$\lambda_1(L_g)vol(M,g)^{2/n}=Y(M^n,[g]).$$

Hence $I(M,g,|A|^2)=Y(M^n ,[g])$.

In particular, $s_g=Y(M,[g])vol(M^n,g)^{-2/n}=I(M,g,|A|^2)|A|^2$ and $u_1$ is a positive constant.
\end{proof}

Now we have that the scalar curvatures  $s_g,s_{g_1}$ and $s_{g_2}$ are constants.

Similar to the classical twistor spinor (see for instance \cite{Baum}), the equation $T\varphi=0$ implies that $$D^2_{(i)}=\frac{n_i}{4(n_i-1)}s_{g_i}\varphi,$$
for $i=1,2$.

Therefore $$D^2=\frac{n_2}{4(n_2-1)}s_{g}\varphi.$$

Let $$\mu=\frac{n_2}{4(n_2-1)}s_{g}.$$

It is easy to see that $\psi:=\sqrt{\mu}\varphi+D\varphi$ is an eigenspinor of the Dirac operator associated to the eigenvalue $\sqrt{\mu}$.  

In \cite{Alex} Alexandrov proved that the existence of such eigenspinor implies that each factor of the product $M=M_1\times M_2$  satisfy that $M_2$ admits a non trivial real Killing spinor and $M_1$ admits a non trivial parallel spinor (when $n_1<n_2$) or $M_1$ admits a non trivial real Killing spinor or a non trivial parallel spinor (when $n_1=n_2$)
 
Hence we have proved that 
\begin{theorem}
If equality $$\lambda^2=\frac{n_2}{4(n_2-1)}Y(M^n,[g])$$ in previous theorem holds then the product metric $g$ is a Yamabe metric, the vector potential $A$ has constant length , $(M_2,g_2)$ admits a non trivial real Killing spinor and $(M_1,g_1)$ admits a parallel spinor (if $n_1<n_2$) or either a parallel spinor or a real Killing spinor when $n_1=n_2$.
\end{theorem}

\section{Proof of Theorem 1.6}
The proof of next theorem is similar to the one in previous section. However, we consider this case separately because the equality part can be characterized completely.

First, we obtain the inequality \eqref{inequality2}. We include the proof for completeness
\begin{theorem}
Let $(M^n=M_1^{n_1}\times M_2^{n_2},g=g_1+g_2,\sigma)$ be a product of closed spin manifolds of dimensions $n_1\leq n_2$. Assume that the scalar curvature $s_g>0$. Let $(\varphi,f)$ be a non trivial solution of the zero mode-type equation $$D_g\varphi=\lambda f\varphi,\quad\int_M f^ndv_g=1. $$
If $|\varphi|^2$ is non-increasing with respect to $f$ then $$\lambda^2\geq \frac{n_2}{4(n_2-1)}Y(M^n,[g]).$$
\end{theorem}

\begin{proof}
In the proof of  Theorem 3.1 we can replace $|A|^2$ by $f^2$ in the integral identity (*) and obtain the following integral identity (**)
\begin{align*}
0=\int_M\frac{|\varphi|^2}{u}\left(-a_n\Delta u+s_gu+4\left(\frac{1-n_2}{n_2}\right)\lambda^2f^2u\right)+\int_M 4|T\varphi|^2+\int_M 4\left( \frac{1}{n_1}-\frac{1}{n_2}\right)|D_1\varphi|^2\\
+\int_M a_n|\varphi|^2\frac{|\nabla u|^2}{u^2}-\int_Ma_n\left\langle\nabla ln(u),\nabla |\varphi|^2\right\rangle.
\end{align*}
Now consider the constant $I(M,g,f)$ defined by $$I(M,g,f)=\inf_{H^1}\frac{\int_M\left( a_n|\nabla u|^2+s_gu^2\right)dv_g}{\int_M (fu)^2dv_g}.$$

By assumption $$-\int_Ma_n\left\langle\nabla ln(u_1),\nabla |\varphi|^2\right\rangle\geq 0,$$
where $u_1$ is the positive function that realizes $I(M,g,f)$.
In particular, inserting $u_1$ in the integral identity (**) we get $$4\left(\frac{n_2-1}{n_2}\right)\lambda^2\geq I(M,g,f).$$
Note from normalization of $f$ and Hölder inequality in $\int_M (fu)^2dv_g$ that

 $$I\geq \inf \frac{\int_M\left( a_n|\nabla u|^2+s_gu^2\right)dv_g}{\left(\int_Mu^{2n/(n-2)}dv_g\right)^{(n-2)/n}}=Y(M^n,[g]). $$ 
Therefore $$\lambda^2\geq \frac{n_2}{4(n_2-1)}Y(M^n,[g]).$$

\end{proof}
 
Now, we discuss the equality case in \eqref{inequality2}.
If $$\lambda^2=\frac{n_2}{4(n_2-1)}Y(M^n,[g])$$ then
$$\lambda^2=\frac{n_2}{4(n_2-1)}I(M,g,f)$$ and $I(M,g,f)=Y(M,[g])$.

The equality $\lambda^2=\frac{n_2}{4(n_2-1)}I(M,g,f)$ and the integral identity (**) imply that $T\varphi=0,\quad D_1\varphi=0$ (if $n_1<n_2$) and $u_1$ is a constant.

The equality $I(M,g,f)=Y(M,[g])$  and $u_1$ constant occur if and only if  $f$ is a constant function and $g$ is a Yamabe metric. The proof is exactly the same of Claim \ref{Claim1}.

Thus $\varphi$ is a Dirac eigenspinor. Therefore the equation $T\varphi=0$ is equivalent to the condition that each factors of the product $M=M_1\times M_2$  satisfy that $M_2$ admits a non trivial real Killing spinor and $M_1$ admits a non trivial parallel spinor (when $n_1<n_2$) or $M_1$ admits a non trivial real Killing spinor or a non trivial parallel spinor (when $n_1=n_2$)
 
Hence we can prove that 
\begin{theorem}
The equality $$\lambda^2=\frac{n_2}{4(n_2-1)}Y(M^n,[g])$$ in previous theorem holds if and only if $g$ is a Yamabe metric, the potential $f$ is constant, $(M_2,g_2)$ admits a non trivial real Killing spinor and $(M_1,g_1)$ admits a parallel spinor (if $n_1<n_2$) or either a parallel spinor or a real Killing spinor when $n_1=n_2$.
\end{theorem}

\begin{proof}
The necessary part for the equality corresponds to the previous discussion. Now assume that $(M_1,g_1)$, $(M_2,g_2)$ are as before and $g=g_1+g_2$ is a Yamabe metric.

By characterization of Dirac eigenvalues on the product $(M=M_1\times M_2,g)$ there exists a spinor field $\varphi$ such that $$D_g\varphi=\lambda\varphi,$$
where $\lambda^2=\frac{n_2}{4(n_2-1)}s_g$.

Let $f=\frac{1}{(vol(M,g)^{1/n}}$. Thus $(\varphi,f)$ is a zero mode of the equation $$D\varphi=\left(\frac{\lambda}{f}\right)f\varphi.$$

Moreover, applying Hölder inequality to estimate $\int_M(fu)^2$ we have that 
\begin{align*}
Y(M^n,[g])=\inf_u \dfrac{\int_M(a_n|\nabla u|^2+s_gu^2)}{\left(\int_Mu^{2n/(n-2)}\right)^{(n-2)/n}}&\leq I(M,g,f)=\inf_u \dfrac{\int_M(a_n|\nabla u|^2+s_gu^2)}{\left(\int_M(fu)^2\right)}\\&\leq s_gf^{-2}= s_g vol(M^n,g)^{2/n}=Y(M^n,[g]). 
\end{align*}
In the last equality we used that $g$ is a Yamabe metric.

Therefore a constant function $u_1$ is the first eigenfunction of the problem $$-a_n\Delta u+s_gu=I(M,g,f)f^2u.$$
In particular $|\varphi|^2$ is non increasing along $f^2$.

Thus $(\varphi,f)$ is a zero mode for the parameter $\lambda/f$ which satisfies the equality in \eqref{inequality2}.

\end{proof} 
 
\begin{remark}
Notice that the product spin manifold $(\mathbb{S}^2\times\mathbb{S}^2,g_0+g_0)$ is one example that satisfies the equality in \eqref{inequality2}. 
\end{remark}

\end{document}